\documentclass[12pt]{article}

\usepackage[bb=fourier,cal=euler,scr=rsfs]{mathalfa}	
\usepackage{enumitem}       
\usepackage{graphicx}
\usepackage{indentfirst}
\usepackage{color}
\usepackage{xcolor}
\usepackage{hyperref}
\usepackage{mathtools}      
\usepackage{mathabx}        
\usepackage{amsthm}

\newtheorem{theoremalph}{Theorem}

\newtheorem{Coro}[theoremalph]{Corollary}
\newtheorem{Theorem}{Theorem}[section]
\newtheorem*{Theorem A}{Theorem A}
\newtheorem*{Theorem A'}{Theorem A'}

\newtheorem{Definition}[Theorem]{Definition}
\newtheorem{Proposition}[Theorem]{Proposition}
\newtheorem{Lemma}[Theorem]{Lemma}

\newtheorem{Question}{Question}

\newtheorem{Remark-numbered}{Remark}
\newtheorem{Corollary}[Theorem]{Corollary}

\newtheorem*{Claim}{Claim}
\newtheorem{Claim-numbered}{Claim}

 \def\NN{{\mathbb N}} 

 \def\RR{{\mathbb R}}

  \def\cG{{\cal G}}  
    
\def\cC{{\cal C}}    \def\cU{{\cal U}}

    \def\cX{{\cal X}}

\def\dim{\operatorname{dim}}

\def\diam{\operatorname{Diam}}
\def\Sing{\operatorname{Sing}}

\begin{document}

\title{Robust transitivity of singular hyperbolic attractors}

\author{Sylvain Crovisier\footnote{S.C. was partially supported by  the ERC project 692925 \emph{NUHGD}.} \and Dawei Yang\footnote{D.Y.  was partially supported by NSFC (11822109, 11671288, 11790274, 11826102).}}

\date{\today}

\maketitle


\abstract{Singular hyperbolicity is a weakened form of hyperbolicity that has been introduced for vector fields in order to allow non-isolated singularities inside the non-wandering set. A typical example of a singular hyperbolic set is the Lorenz attractor.
However, in contrast to uniform hyperbolicity, singular hyperbolicity does not immediately imply robust topological properties, such as the transitivity.

In this paper, we prove that open and densely inside the space of $C^1$ vector fields of a compact manifold,
any singular hyperbolic attractors is robustly transitive.}

\section{Introduction}
Lorenz \cite{Lo} in 1963 studied some polynomial ordinary differential equations in $\RR^3$. He found some strange attractor with the help of computers. By trying to understand the chaotic dynamics in Lorenz' systems, \cite{ABS, Gu, GW} constructed some geometric abstract models which are called \emph{geometrical Lorenz attractors}: these are robustly transitive non-hyperbolic chaotic attractors with singularities in three-dimensional manifolds.

In order to study attractors containing singularities for general vector fields, Morales-Pacifico-Pujals \cite{MPP} first gave the notion of \emph{singular hyperbolicity} in dimension 3. This notion can be adapted to the higher dimensional case, see \cite{BD,CDYZ,MeM08,ZGW08}. 
In the absence of singularity, the singular hyperbolicity coincides with the usual notion of uniform hyperbolicity; in that case it has many nice dynamical consequences: spectral decomposition, stability, probabilistic description,...
But there also exist open classes of vector fields exhibiting singular hyperbolic attractors with singularity:
the geometrical Lorenz attractors are such examples.
In order to have a description of the dynamics of general flows,
we thus need to develop a systematic study of the singular hyperbolicity in the presence of singularity. This paper contributes to that goal.

We do not expect to describe the dynamics of arbitrary vector fields.
Instead, one considers the Banach space ${\cal X}^r(M)$ of  all $C^r$ vector fields on a compact manifold $M$
without boundary
and focus on a subset $\cG$ which is dense and as large as possible.
A successful approach consists in considering subsets that are $C^1$-residual (i.e. containing a dense G$_\delta$ subset
with respect to the $C^1$-topology), but this does not handle immediately smoother systems.
For that reason it is useful to work with subsets $\cG\subset {\cal X}^1(M)$ that are $C^1$-open and $C^1$-dense and
to address for each dynamical property the following question:
\smallskip

\emph{Knowing that a given property holds on a $C^1$-residual subset of vector fields,
is it satisfied on a $C^1$-open and dense subset?}
\medskip

\paragraph{\bf Precise setting.}
Given a vector field $X\in{\cal X}^1(M)$, the flow generated by $X$ is denoted by $(\varphi_t^X)_{t\in \RR}$, and sometimes by $(\varphi_t)$ if there is no confusion.  A point $\sigma$ is a \emph{singularity} of $X$ if $X(\sigma)=0$. A point $p$ is \emph{periodic} if it is not a singularity and there is $T>0$ such that $\varphi_T(p)=p$. 
We denote by ${\rm Sing}(X)$ the set of singularities and by ${\rm Per}(X)$ the set of periodic orbits of $X$. The union ${\rm Crit}(X):={\rm Sing}(X)\cup{\rm Per}(X)$
is the set of \emph{critical elements} of $X$.

We will mainly discuss the recurrence properties of the dynamics. An invariant compact set $\Lambda$ is
\emph{transitive} if it contains a point $x$ whose positive orbit is dense in $\Lambda$.
More generally $\Lambda$ is \emph{chain-transitive} if for any $\varepsilon>0$ and $x,y\in \Lambda$,
there exists $x_0=x, x_1,\dots, x_n=y$ in $\Lambda$ and $t_0, t_1,\dots, t_{n-1}\ge1$ such that
$d(x_{i+1}, \varphi_{t_i}(x_i))<\varepsilon$ for each $i=0,\dots, n-1$. A compact invariant set $\Lambda$ is said to be a \emph{chain-recurrence class} if it is chain-transitive, and is not a proper subset of any other chain-transitive compact invariant set.
The chain-recurrence classes are pairwise disjoint.

Among invariant sets, important ones are those satisfying an attracting property.
An invariant compact set $\Lambda$ is an \emph{attracting set} is there exists a neighborhood $U$ such that
$\cap_{t>0} \varphi_t(U)=\Lambda$ and an \emph{attractor} if it is a transitive attracting set. More generally $\Lambda$ is \emph{Lyapunov stable} if for any neighborhood $V$
there exists a neighborhood $U$ such that $\varphi_t(U)\subset V$ for all $t>0$.

We will also study the stability of these properties. For instance a set $\Lambda$ is \emph{robustly transitive}
if there exist neighborhoods $U$ of $\Lambda$ and $\cU$ of $X$ in $\cX^1(M)$ such that for any $Y\in \cU$,
the maximal invariant set $\cap _{t\in \RR} \varphi^Y_t(\overline U)$ is transitive, and coincides with $\Lambda$ when $Y=X$.
\medskip

As said before, singular hyperbolicity is a weak notion of hyperbolicity that has been introduced in order to characterize some robust dynamical properties.
A compact invariant set $\Lambda$ is \emph{singular hyperbolic} if for the flow $(\varphi_t)$ generated by either $X$ or $-X$, there are a continuous $D\varphi_t$-invariant splitting $T_\Lambda M=E^{ss}\oplus E^{cu}$ and $T>0$ such that for any $x\in\Lambda$:
\begin{itemize}
\item $E^{ss}$ is contracted: $\|D\varphi_T|_{E^{ss}(x)}\|\le 1/2$ .
\item $E^{ss}$ is dominated by $E^{cu}$: $\|D\varphi_T|_{E^{ss}(x)}\|\|D\varphi_{-T}|_{E^{cu}(\varphi_T(x))}\|\le 1/2$.
\item $E^{cu}$ is area-expanded: $|{\det}D\varphi_{-T}|_{P}|\le 1/2$ for any two-dimensional plane $P\subset E^{cu}(x)$.
\end{itemize}
Note that if $\Lambda$ is a singular hyperbolic set, and $\Lambda\cap {\rm Sing}(X)=\emptyset$, then $\Lambda$ is a hyperbolic set.

Some robust properties or generic assumptions imply the singular hyperbolicity.
In dimension $3$, robustly transitive sets are singular hyperbolic~\cite{MPP} and
any generic vector field $X\in \cX^1(M)$ far from homoclinic tangencies supports a global singular hyperbolic structure~\cite{CY1,CY2}.
In higher dimension, the transitive attractors of generic vector field $X\in \cX^1(M)$ satisfying the star property are singular hyperbolic~\cite{SGW}.
\medskip

\paragraph{\bf Statement of the results.}
It is well known that for uniformly hyperbolic sets:
\begin{itemize}
\item  chain-transitivity and local maximality
(i.e. the set coincides with the maximal invariant set in one of its neighborhoods)
imply the robust transitivity;
\item Lyapunov stability implies that the set is an attracting set.
\end{itemize}
We do not know whether these properties extend to general singular hyperbolic sets but this has been proved in~\cite{pyy}
for $C^1$-generic vector fields: generically Lyapunov stable chain-recurrence classes which are singular hyperbolic are
transitive attractors.
We show that this holds robustly.

\begin{theoremalph}\label{Thm:main}
There is an open and dense set ${\cal U}\in{\cal X}^1(M)$ such that for any $X\in{\cal U}$, any singular hyperbolic Lyapunov stable chain-recurrence class $\Lambda$ of $X$ is a robustly transitive attractor.
\end{theoremalph}

Let us mention that a more general notion of hyperbolicity, called \emph{multi-singular hyperbolicity} has been recently introduced
in order to characterize star vector fields, in~\cite{BD}  (see also~\cite{CDYZ}).
In contrast to Theorem~\ref{Thm:main} above, a multi-singular chain-recurrence class
of a $C^1$-generic vector field may be isolated and not robustly transitive~\cite{dL}.

Under the setting of Theorem~\ref{Thm:main}, we have a more accurate description of the singular hyperbolic attractors in Theorem~\ref{Thm:main}. Two hyperbolic periodic orbits $\gamma_1$ and $\gamma_2$ are \emph{homoclinically related} if $W^s(\gamma_1)$ intersects $W^u(\gamma_2)$ transversely and $W^s(\gamma_2)$ intersects $W^u(\gamma_1)$ transversely.
The \emph{homoclinic class} $H(\gamma)$ of a hyperbolic periodic orbit $\gamma$ is the closure of the union of the periodic orbits that are
homoclinically related to $\gamma$. This is a transitive invariant compact set.
In dimension $3$, any singular hyperbolic transitive attractor is a homoclinic class~\cite{APu}; we show that this also holds in higher dimension
for vector fields in a dense open set.

\begin{theoremalph}\label{Thm:homoclinic-related}
There is an open and dense set ${\cal U}\in{\cal X}^1(M)$ such that for any $X\in{\cal U}$, any singular hyperbolic Lyapunov stable chain-recurrence class $\Lambda$ of $X$ (not reduced to a singularity) is a homoclinic class
(in particular the set of periodic points is dense in $\Lambda$).

Moreover, any two periodic orbits contained in $\Lambda$ are hyperbolic and homoclinically related.
\end{theoremalph}

In dimension 3, we know more properties of chain-recurrence classes with singularities for generic systems. This gives the following consequence:

\begin{Coro}\label{Cor:dimension3}

When $\dim M=3$, there is an open dense set ${\cal U}\in{\cal X}^1(M)$ such that for any $X\in\cU$, any singular hyperbolic chain-recurrence class is robustly transitive.
Unless the class is an isolated singularity, it is a homoclinic class.

\end{Coro}

Theorem~\ref{Thm:main} in fact solves Conjecture 7.5 in \cite{BM}: \emph{a $C^1$ generic three-dimensional flow has either infinitely many sinks or finitely many \emph{robustly transitive} attractors (hyperbolic or singular hyperbolic ones) whose
basins form a full Lebesgue measure set of $M$.}
With~\cite{M}, only the robustness was unknown before this paper.
\medskip

\paragraph{\bf Further discussions.}
It is natural to expect that the previous results hold for arbitrary singular hyperbolic chain-recurrence class,
also in higher dimension:

\begin{Question}
Does there exists (when $\dim(M)\geq 4$) a dense and open subset $\cU\subset \cX^1(M)$ such that for any $X\in \cU$,
any singular hyperbolic chain-recurrence class is robustly transitive? is a homoclinic class?
\end{Question}
This would imply in particular that if $X\in \cU$ is singular hyperbolic (i.e. each of its chain-recurrence class is singular hyperbolic),
it admits finitely many chain-recurrence classes only.
\medskip

One can also study stronger forms of recurrence.
It is known that for $C^r$-generic vector field, each homoclinic class is topologically mixing, see~\cite{Abdenur-Avila-Bochi}
and one may wonder if this holds robustly.

\begin{Question}
Does there exists a dense and open subset $\cU\subset \cX^1(M)$ such that for any $X\in \cU$,
any singular hyperbolic transitive attractor is robustly topologically mixing.
\end{Question}

The answer is positive in the case of non-singular transitive attractors~\cite{field-melbourne-torok}.
Also \cite{araujo-melbourne} proves that $\cX^2(M)$ contains a $C^2$-dense and $C^2$-open subset of vector fields
such that any singular hyperbolic robustly transitive attractor with $\dim(E^{cu})=2$ is robustly topologically mixing.


\section{Preliminaries}

%
%
%

\subsection{Local chain-recurrence classes and homoclinic classes}
Let us consider a $C^1$ vector field $X$ and a compact subset $K$ of $M$.
We will use the following notations. The orbit of a point $x$ under the flow of $X$ is denoted by ${\rm Orb}(x)$.
Two injectively immersed sub-manifolds $W_1$ and $W_2$ are said to intersect transversely at a point $x\in W_1\cap W_2$ if $T_x W_1+T_x W_2=T_x M$. The set of the transverse intersections is denoted by $W_1\pitchfork W_2$.


\paragraph{Hyperbolicity.} A compact invariant set $\Lambda$ of a vector field $X$ is \emph{hyperbolic} if there is a continuous invariant splitting
$$T_\Lambda M=E^{ss}\oplus \RR.X\oplus E^{uu},$$
and constants $C,\lambda>0$ such that for any $x\in\Lambda$ and any $t\ge 0$,
$$\|D\varphi_t|_{E^{ss}(x)}\|\le C{\rm e}^{-\lambda t} \text{ and } \|D\varphi_{-t}|_{E^{uu}(x)}\|\le C{\rm e}^{-\lambda t}.$$
Note that $\dim \RR.X(x)=0$ when $x$ is a singularity.

To each point $x$ in a hyperbolic set $\Lambda$ is associated a strong stable manifold $W^{ss}(x)$ and a strong unstable manifold $W^{uu}(x)$ that are tangent to $E^{ss}$ and $E^{uu}$, respectively. For a hyperbolic critical element $\gamma$ we introduce
the stable and unstable sets $W^s(\gamma)$ and $W^u(\gamma)$.
Note that:
\begin{itemize}
\item for a hyperbolic singularity $\sigma$, we have $W^s(\sigma)=W^{ss}(\sigma)$ and $W^u(\sigma)=W^{uu}(\sigma)$;

\item for a hyperbolic periodic orbit $\gamma$, the sets $W^s(\gamma)$ and $W^u(\gamma)$ are injective immersed sub-manifolds and
 $$W^{s}(\gamma)=\bigcup_{x\in\gamma}W^{ss}(\gamma),~~~W^{u}(\gamma)=\bigcup_{x\in\gamma}W^{uu}(\gamma).$$
\end{itemize}

\paragraph{Homoclinic classes.}
The \emph{homoclinic class} of a hyperbolic periodic orbit $\gamma$ is
$$H(\gamma)={\overline {W^s(\gamma)\pitchfork W^u(\gamma)}}.$$

The set $H(\gamma)$ can also be defined in another way. Two hyperbolic periodic orbits $\gamma_1$ and $\gamma_2$ are said to be \emph{homoclinically related} if
$W^s(\gamma_1)\pitchfork W^u(\gamma_2)$ and $W^s(\gamma_2)\pitchfork W^u(\gamma_1)$ are non-empty.
The Birkhoff-Smale theorem implies that
this is an equivalence relation and that $H(\gamma)$
is the union of the hyperbolic periodic orbits that are homoclinically related
with $\gamma$. Moreover $H(\gamma)$ is a transitive set.
See~\cite{newhouse} and~\cite[Section 2.5.5]{APa}.
%
%
%

\paragraph{Chain-recurrence classes.}
If $\sigma$ is a critical element of $X$, we denote by $C(\sigma)$ the chain-recurrent class of $X$ that contains $\sigma$.

\paragraph{Singular hyperbolicity.}
%
As for the uniform hyperbolicity, the singular hyperbolicity is robust:
there exist a $C^1$-neighborhood $\cU$ of $X$
and a neighborhood $U$ of $\Lambda$ such that for any $Y\in \cU$,
any $\varphi^Y$-invariant compact set
$\Lambda'$ contained in $U$ is singular hyperbolic.

In this case, any point $x\in \Lambda$ admits a strong stable manifold $W^{ss}(x)$ tangent to $E^{ss}(x)$.
Moreover one can choose a neighborhood $W^{ss}_{loc}(x)$ of $x$ in $W^{ss}(x)$ which is an embedded disc which varies continuously
for the $C^1$-topology with respect to $x$ and to the vector field.

\subsection{Genericity}

We recall several generic results.  Some notations and definitions are given first.
\begin{itemize}

\item For a hyperbolic critical element $\gamma$ of $X$, its hyperbolic continuation will be denoted by $\gamma_Y$ for $Y$ that is $C^1$-close to $X$.

\item A compact invariant set $\Lambda$ is \emph{locally maximal} if there is a neighborhood $U$ of $\Lambda$ such that $\Lambda=\cap_{t\in\RR}\varphi_t(U)$.

\end{itemize}

\begin{Proposition}\label{Pro:generic}
There is a dense $G_\delta$ set $\mathcal G$ in ${\mathcal X}^1(M)$ such that for any $X\in\mathcal G$, we have:
\begin{enumerate}

\item\label{i.KS} $X$ is Kupka-Smale: each critical element of $X$ is hyperbolic;
the intersections of the stable manifold of one critical element with the unstable manifold of another critical critical element
are all transverse.

\item\label{i.continuous} For any $\sigma\in {\rm Crit}(X)$, the map $Y\mapsto C(\sigma_Y)$ is continuous at $X$ for the Hausdorff topology.


\item\label{i.chain-homo} If a chain-recurrence class contains a hyperbolic periodic orbit $\gamma$, then it coincides with the homoclinic class $H(\gamma)$ of $\gamma$.

\item\label{i.isolated} If a chain-recurrence class of a critical element $\sigma$ is locally maximal, then it is robustly locally maximal: there are a neighborhood $\cU$ of $X$
in ${\mathcal X}^1(M)$ and a neighborhood $U$ of $C(\sigma)$ such that $C(\sigma_Y)$ is the maximal invariant set in $U$ for any $Y\in\cU$.

\item\label{i.transverse-close} For any critical elements $\gamma_1,\gamma_2$ in a same chain-recurrence class $C(\gamma_1)$ and satisfying $\dim W^u(\gamma_2)\ge\dim W^u(\gamma_1)$, then any neighborhood $U_x$ of a point $x\in C(\gamma_1)\cap W^s_{loc}(\gamma_1)$ contains a point $y\in W^u(\gamma_2)\pitchfork W^s_{loc}(\gamma_1)$.

\item\label{i.one-class-related} Any hyperbolic periodic orbits in a same chain-recurrence class and with the same stable dimension are homoclinically related.

\item\label{i.unstable-Lyapunovstable}
For $\gamma\in {\rm Crit}(X)$, if $C(\gamma)$ contains the unstable manifold $W^u(\gamma)$, then it is
Lyapunov stable and $\overline{W^u(\gamma)}=C(\gamma)$.
Moreover $W^u(\gamma_Y)\subset C(\gamma_Y)$ still holds for any $Y$ that is $C^1$-close to $X$.

\item\label{i.dim3}
If $\sigma\in \mathrm{Sing}(X)$ has unstable dimension equal to one, then either $C(\sigma)=\{\sigma\}$
or $C(\sigma)$ is Lyapunov stable.
\end{enumerate}

\end{Proposition}
The properties in Proposition~\ref{Pro:generic} are well known. We give some comments:
\begin{itemize}
\item[--] {Item~\ref{i.KS}} is the classical Kupka-Smale theorem \cite{kupka,smale}.
\item[--] {Item~\ref{i.continuous}} follows from the upper semi-continuity of the chain-recurrence class $C(\sigma_Y)$ with respect to the vector field:
the continuity holds at generic points.
\item[--] {Item~\ref{i.chain-homo} and Item~\ref{i.isolated}} have been proved for diffeomorphisms in \cite[Remarque 1.10 and Corollaire 1.13]{BC}. The proofs for flows are similar.
\item[--] {Item~\ref{i.transverse-close}} is an application of the connecting lemma in \cite{BC}. For any point $x$ as in the statement, by using the connecting lemma, there is a point $y$ close to $x$ such that $y\in W^u(\gamma_2)\pitchfork W^s_{loc}(\gamma_1)$ for $Y$ close to $X$. Then one can apply a Baire argument to conclude.
\item[--] {Item~\ref{i.one-class-related}} is a consequence of Item~\ref{i.transverse-close}.
\item[--] {Item~\ref{i.unstable-Lyapunovstable}} and~\ref{i.dim3} are applications of the connecting lemma for pseudo-orbits in \cite{BC}, see also \cite[Lemmas 3.13, 3.14 and 3.19]{GY}.
\end{itemize}
\medskip

We know the following theorem from \cite[Corollary C]{pyy} and \cite[Theorem 1.1]{ALM}
(and previously \cite[Theorem D]{MP} in the three-dimensional case).

\begin{Theorem}\label{Thm:periodic-contained}
There is a dense $G_\delta$ set ${\cal G}\subset {\cal X}^1(M)$ such that for any vector field $X\in{\cal G}$, if $C(\sigma)$ is a singular hyperbolic Lyapunov stable chain-recurrence class (and not reduced to $\sigma$), then $C(\sigma)$ contains periodic orbits and is an attractor.
\end{Theorem}

The following proposition gives some open and dense properties  for chain-recurrence classes.

\begin{Proposition}\label{Pro:robust}
There is an open and dense set $\cU\subset \cX^1(M)$ such that any $X\in\cU$ has a neighborhood $\cU_X$
with the following property.
For each $\sigma\in {\rm Sing}(X)$ and $Y\in \cU_X$,
\begin{enumerate}
\item \label{Item:Lyapunov-singular-class} $W^u(\sigma_X)\subset C(\sigma_X) \Leftrightarrow W^u(\sigma_Y)\subset C(\sigma_Y)$;
\item \label{Item:non-trivial---} $C(\sigma_X)$ is non-trivial ({i.e.} not reduced to $\{\sigma_X\}$) if and only if $C(\sigma_Y)$ is non-trivial.
\end{enumerate}
\end{Proposition}

\begin{proof}
We take the dense $G_\delta$ set $\cG$ provided by Proposition~\ref{Pro:generic} and we consider the subset $\cG_n\subset \cG$ of vector fields with exactly $n$ singularities. Consider $X\in \cG_n$ and denote the set of its singularities by $\{\sigma_1,\cdots,\sigma_n\}$.

Note that ${\rm Closure}(W^u(\gamma_X))$ varies lower semi-continuously with respect to the vector field $X$ and $C(\gamma_X)$ varies upper semi-continuously with respect to the vector field $X$. So, for any (hyperbolic) critical element $\gamma$ of $X$, if $W^u(\gamma_X)\setminus C(\gamma)\neq\emptyset$, then there is a neighborhood $\cU_X$ such that $W^u(\gamma_Y)\setminus C(\gamma_Y)\neq\emptyset$.
Similarly, if $C(\sigma_X)=\{\sigma_X\}$, there exists a neighborhood $\cU_X$ such that $C(\sigma_Y)$ is contained in a small neighborhood of $C(\sigma_X)$.
Since $\sigma_X$ is hyperbolic, $C(\sigma_Y)=\{\sigma_Y\}$.

Together with the Items~\ref{i.continuous} and \ref{i.unstable-Lyapunovstable} of Proposition~\ref{Pro:generic},
for each $\sigma_i$, there is an open set $\cU_{i,X}$ such that
\begin{enumerate}
\item either for any $Y\in\cU_{i,X}$ we have $W^u(\sigma_{i,Y})\setminus C(\sigma_{i,Y})\neq\emptyset$, or for any $Y\in\cU_{i,X}$ we have $W^u(\sigma_{i,Y})\subset C(\sigma_{i,Y})$;
\item $C(\sigma_i)$ is non-trivial if and only if $C(\sigma_{i,Y})$ is non-trivial for any $Y\in\cU_{i,X}$.
\end{enumerate}

By reducing $\cU_{i,X}$ if necessary for each $i$, any singularity of $Y\in \cap_{i=1}^n\cU_{i,X}$ is the continuation of a singularity of $X$. Now we take

$$\cU=\bigcup_{n\in\NN}\cU_n,~~~\textrm{where~}\cU_n=\bigcup_{X\in\cG_n}(\cap_{i=1}^n\cU_{i,X}).$$

It is clear that $\cG\subset\cU$. Thus $\cU$ is dense. Now we check that the open set $\cU$ has the required properties. For any vector field $Y\in\cU$, there is $n\in\NN$ and a vector field $X\in\cG_n$ such that $Y\in\cap_{i=1}^n\cU_{i,X}$. Thus, any singularity $\sigma_Y$ of $Y$ is a continuation of a singularity $\sigma_{i,X}$ of $X$. By  the choice of $\cU_{i,X}$, we have
\begin{enumerate}

\item $W^u(\sigma_Y)\subset C(\sigma_Y)$ if and only if $W^u(\sigma_Z)\subset C(\sigma_Z)$ for any $Z\in \cap_{i=1}^n\cU_{i,X}$.

\item $C(\sigma_Y)$ is non-trivial if and only if $C(\sigma_{i,Z})$ is non-trivial for any $Z\in \cap_{i=1}^n\cU_{i,X}$.

\end{enumerate}
This implies the proposition.
\end{proof}

\subsection{Robust properties of singular hyperbolic attractors}

\begin{Proposition}\label{Pro:strong-stable-Y}
Assume that $C(\sigma)$ is a transitive singular hyperbolic attractor of a vector field $X$ containing a hyperbolic singularity $\sigma$
(and not reduced to $\sigma$). Then  there is a neighborhood ${\cal U}_X$ of $X$ such that the continuation $C(\sigma_Y)$ of any $Y\in\cU_X$ satisfies:
\begin{enumerate}

\item\label{i.splitting-type-Y} $C(\sigma_Y)$ admits a singular hyperbolic splitting $T_{C(\sigma_Y)}M=E^{ss}_{Y}\oplus E^{cu}_{Y}$ for $Y$ such that $E^{ss}_{Y}$ is $D\varphi^Y_T$-contracted and $E^{cu}$ is $D\varphi^Y_T$-area-expanded for some $T>0$.

\item\label{i.flow-direction-Y} For each $x\in C(\sigma_Y)$, we have that $Y(x)\subset E^{cu}_{Y}(x)$.

\item\label{i.strong-stable-outside-Y} The stable space $E^{s}_{Y}(\rho)$ of any singularity $\rho\in C(\sigma_Y)$ has a dominated splitting $E^{s}_{Y}(\rho)=E^{ss}_{Y}(\rho)\oplus E^{c}_{Y}(\rho)$ such that $\dim E^{c}_{Y}=1$. Moreover $W^{ss}_{Y}(\rho)\cap C(\sigma_Y)=\{\rho\}$.

%

\end{enumerate}

\end{Proposition}

\begin{proof}

We first prove Item~\ref{i.splitting-type-Y} for the vector field $X$. Let us recall a result from \cite[Lemma 3.4]{BGY}: for a transitive set $\Lambda$ with a dominated splitting $T_\Lambda M=E\oplus F$ we have

\begin{itemize}
\item either, $X(x)\in E(x)$ for any point $x\in\Lambda$,
\item or, $X(x)\in F(x)$ for any point $x\in\Lambda$.
\end{itemize}

Now we consider a transitive singular hyperbolic attractor $C(\sigma)$. Assume by contradiction that it has a singular hyperbolic splitting $T_{C(\sigma)}M=E^{cs}\oplus E^{uu}$ such that $E^{uu}$ is expanded. Since it is an attractor, the strong unstable manifold of $\sigma$ is contained in $C(\sigma)$. For a point $z\in W^{uu}(\sigma)$, we have $X(z)\in E^{uu}(z)$. Thus, for any point $x\in C(\sigma)$, we have $X(x)\in E^{uu}(x)$. But $\RR.X$ cannot be uniformly expanded everywhere and one gets a contradiction. Thus
the singular hyperbolic splitting on $C(\sigma)$ has the form $T_{C(\sigma)}M=E^{ss}\oplus E^{cu}$ and Item~\ref{i.splitting-type-Y} is proved for $X$.
Since the singular hyperbolic splitting is robust, and the chain-recurrence class varies upper semi-continuously, $C(\sigma_Y)$ admits the same kind of splitting
for any $Y$ close to $X$. This proves Item~\ref{i.splitting-type-Y}.
\bigskip

Now we prove Items~\ref{i.flow-direction-Y} and~\ref{i.strong-stable-outside-Y} for the vector field $X$.
By using the result in \cite{BGY} again, if for some regular point $x$, we have $X(x)\in E^{ss}(x)$, then this property holds for every point. But $\RR.X$ cannot be uniformly contracted and one gets a contradiction again. Thus Item~\ref{i.flow-direction-Y} for $X$ is proved.
Since for any point $x\in C(\sigma)$, we have $X(x)\in E^{cu}(x)$, the strong stable manifold $W^{ss}(\sigma)$ cannot contain a regular orbit. Thus $W^{ss}(\sigma)\cap C(\sigma)=\{\sigma\}$. On the other hand, the stable manifold $W^s(\sigma)$ contains regular orbits of $C(\sigma)$ since $C(\sigma)$ is transitive. Thus $\dim E^s(\sigma)>\dim E^{ss}(\sigma)$. If $\dim E^s(\sigma)>\dim E^{ss}(\sigma)+1$, then the dimension of the invariant space $E^c(\sigma):=E^{s}(\sigma)\cap E^{cu}(\sigma)$ is at least 2. But this space is not area-expanded, this contradicts the definition of the singular hyperbolicity. Thus, we have $\dim E^c(\sigma)=1$. The proof of Item~\ref{i.strong-stable-outside-Y} for $X$ is complete.
\bigskip

Now we give the proof of Item~\ref{i.strong-stable-outside-Y} for any $Y$ close to $X$. Since $C(\sigma)$ is an attractor of $X$, there is a neighborhood $U$ of $C(\sigma)$ such that $U\cap {\rm Sing}(X)\subset C(\sigma)$. Consequently, for $Y$ close to $X$, any singularity $\rho_Y\in C(\sigma_Y)$ is a continuation of some singularity $\rho_X$ in $C(\sigma)$. Since Item~\ref{i.strong-stable-outside-Y} holds for $\rho_X$,
for $Y$ close to $X$ the singularity $\rho_Y$ still has a dominated splitting
$E^{s}_{Y}(\rho)=E^{ss}_{Y}(\rho)\oplus E^{c}_{Y}(\rho)$ with
$\dim(E^{c}_Y=1$.

it also holds for every singularity $\rho_Y\in C(\sigma_Y)$ . 

Let us assume by contradiction that there is a sequence of vector fields $(X_n)\to X$ and a singularity $\rho_X\in C(\sigma)$ such that
$$W^{ss}_{X_n}(\rho_{X_n})\cap C(\sigma_{X_n})\setminus\{\rho_{X_n}\}\neq\emptyset.$$
Thus, there is $\varepsilon_0$ that is independent of $n$ such that the local manifold $W^{ss}_{X_n,\varepsilon_0}(\rho_{X_n})$ of size $\varepsilon_0$ intersects $C(\sigma_{X_n})$ for each $n\in\NN$.
Since the homoclinic class $C(\sigma_Y)$ is upper semi-continuous with respect to $Y$,
we have that the boundary of  $W^{ss}_{\varepsilon_0}(\rho)$ intersects $C(\sigma)$. This contradicts Item~\ref{i.strong-stable-outside-Y} for $X$. Thus the proof of Item~\ref{i.strong-stable-outside-Y} of this proposition is complete.
\bigskip

Let us assume by contradiction that Item~\ref{i.flow-direction-Y} fails: this means that there is a sequence of vector fields $X_n\to X$ and a sequence of points $x_n\in C(\sigma_{X_n})$ such that $X_n(x_n)$ is not contained in $E^{cu}_{X_n}(x_n)$.

From the dominated splitting $E^{ss}\oplus E^{cu}$,
by replacing $x_n$ by a backward iterate, one can assume furthermore
that the angle between $\RR.X_n(x_n)$ and $E^{ss}_{X_n}(x_n)$ is less than $1/n$.
Since for regular points $y$ close to a singularity $\rho$, the angle between $\RR.X_n(y)$ and $E^{ss}_{X_n}(y)$ is uniformly bounded away from zero,
the points $x_n$ are far from the singularities. One can thus assume that $(x_n)$ converges to a regular point $x\in C(\sigma)$.
By taking the limit, we have $X(x)\in E^{ss}(x)$. This contradicts the Item~\ref{i.flow-direction-Y} for $X$.
Thus Item~\ref{i.flow-direction-Y} holds for any $Y$ close to $X$.
\end{proof}

\section{Density of the unstable manifold}

This section is devoted to the following result which will be used to prove the density of the unstable manifold.

\begin{Theorem}\label{Thm:unstable-dense}
There is a dense G$_\delta$ set $\cal G$ in ${\cal X}^1(M)$ such that for any $X\in{\cal G}$, for any singular hyperbolic Lyapunov stable chain-recurrence class $C(\sigma)$ of $X$,
and for any hyperbolic periodic orbit $\gamma$ in $C(\sigma)$, there exists a neighborhood ${\cal U}_X$ of $X$ with the following property.

For any $Y\in{\cal U}_X$ and for any $x\in C(\sigma_Y)$, either $x$ belongs to the unstable manifold of a singularity
or $W^{ss}_{loc,Y}(x)$ intersects $W^{u}_{Y}(\gamma_Y)$ transversely. 
\end{Theorem}

Let $\cG$ be a dense G$_\delta$ subset in ${\cal X}^1(M)$ given by Proposition~\ref{Pro:generic}. Let $X\in\cG$.
We first prove some preliminary lemmas.

\begin{Lemma}
For any singularity $\rho\in C(\sigma)$,
and for any $x\in (W^s(\rho)\cap C(\sigma))\setminus \{\rho\}$,
the submanifolds $W^{ss}_{loc}(x)$ and $W^u(\gamma)$ have a transverse intersection point.
\end{Lemma}
\begin{proof}
Consider $x\in (W^s(\rho)\cap C(\sigma))\setminus \{\rho\}$.
Without loss of generality, we may assume that $x$ belongs to $W^s_{loc}(\rho)$.

By the Item~\ref{i.transverse-close} of Proposition~\ref{Pro:generic}, there exists a transverse intersection point $y$ between
$W^{s}_{loc}(\rho)$ and $W^u(\gamma)$ near $x$. The flow preserves $W^u(\gamma)$, $W^s(\rho)$
and the strong stable foliation of $W^s(\rho)$; moreover since the strong stable foliation is one-codimensional inside
$W^s(\rho)$ and since $x,y$ are two points close in $W^s(\rho)\setminus W^{ss}(\rho)$,
there exists $t\in \RR$ such that $\varphi_t(W^{ss}(y))=W^{ss}(x)$. In particular, $\varphi_t(y)$
is a transverse intersection point between $W^{ss}(x)$ and $W^u(\gamma)$.
That point can be chosen arbitrarily close to $x$ in $W^{ss}(x)$, hence belongs to $W^{ss}_{loc}(x)$.
\end{proof}

Since the unstable manifold $W^u(\gamma_Y)$ and the
leaves of the strong stable foliation $W^{ss}_{loc,Y}(x)$ vary continuously for the $C^1$ topology with respect to $x$ and $Y$,
a compactness argument gives:
\begin{Corollary}\label{c:robust-singular}
For any singularity $\rho\in C(\sigma)$,
let $K_\rho$ be a compact subset of $W^s(\rho)\cap C(\sigma))\setminus \{\rho\}$.
Then there exists a neighborhood $V_\rho$ of $K_\rho$ and a $C^1$-neighborhood $\cU_\rho$
of $X$ such that for any $Y\in \cU_\rho$ and any $y\in V_\rho\cap C(\sigma_Y)$,
the submanifolds $W^{ss}_{loc, Y}(y)$ and $W^u(\gamma_Y)$ have a transverse intersection point.
\end{Corollary}

\begin{Lemma}\label{l:robust-regular}
For any invariant compact set $K_\mathrm{reg}\subset C(\sigma)\setminus \mathrm{Sing}(X)$,
there exists a neighborhood $V_\mathrm{reg}$ of $K_\mathrm{reg}$ and a $C^1$-neighborhood $\cU_\mathrm{reg}$
of $X$ such that for any $Y\in \cU_\mathrm{reg}$ and any $y\in V_\mathrm{reg}\cap C(\sigma_Y)$,
the submanifolds $W^{ss}_{loc, Y}(y)$ and $W^u(\gamma_Y)$ have a transverse intersection point.
\end{Lemma}
\begin{proof}
Since $C(\sigma)$ is singular hyperbolic, the invariant compact set $K_\mathrm{reg}\subset C(\sigma)\setminus \mathrm{Sing}(X)$
is hyperbolic. By the shadowing lemma, and a compactness argument, there exist finitely many
periodic orbits $\gamma_1,\dots,\gamma_\ell$ in an arbitrarily small neighborhood of $K_\mathrm{reg}$
such that for any $y\in K_\mathrm{reg}$, the submanifolds $W^{ss}_{loc}(y)$ intersects transversally
some $W^u(\gamma_i)$. Moreover each $\gamma_i$ is in the chain-recurrence class of a point
of $K_\mathrm{reg}$, hence is contained in $C(\sigma)$. By the Item~\ref{i.one-class-related} of Proposition~\ref{Pro:generic}, $\gamma_i$ is homoclinically related to $\gamma$.
The inclination lemma then implies that for any $y\in K_\mathrm{reg}$, the submanifolds $W^{ss}_{loc}(y)$ intersects transversally
$W^u(\gamma)$. As before this property is $C^1$-robust.
\end{proof}

\begin{proof}[Proof of Theorem~\ref{Thm:unstable-dense}]
Let $\rho_1,\dots,\rho_\ell$ be the singularities contained in $C(\sigma)$.
For each of them, one chooses a compact set $K_i\subset W^s(\rho_i)\cap C(\sigma))\setminus \{\rho_i\}$
which meets each orbit of $W^s(\rho_i)\cap C(\sigma))\setminus \{\rho_i\}$.
The Corollary~\ref{c:robust-singular} associates open sets $V_i$ and $\cU_i$.
There exists a neighborhood $O_i$ of $\rho_i$ such that any point
$z\in O_i\setminus W^u_{loc}(\rho_i)$ has a backward iterate in a compact subset of $V_i$.

Let $K_\mathrm{reg}$ be the maximal invariant set of $C(\sigma)\setminus \cup_i (V_i\cup O_i)$.
The Lemma~\ref{l:robust-regular} associates open sets $\cU_\mathrm{reg}$ and $V_\mathrm{reg}$.
By construction, any point $z\in C(\sigma)$ either belongs to some $W^u(\rho_i)$,
has a backward iterate in a compact subset of some $V_i$,
or belongs to $K_\mathrm{reg}$.
Since $C(\sigma_Y)$ varies upper semi-continuously with $Y$,
this property is still satisfied  for $Y$ in a $C^1$-neighborhood $\cU_\mathrm{return}$ of $X$.

We set $\cU_X=\cU_1\cap\dots\cap\cU_\ell\cap\cU_\mathrm{reg}\cap \cU_\mathrm{return}$.
For any $Y\in \cU_X$ and any $x\in C(\sigma)$ which does not belong to the unstable manifold of a singularity
$\rho_{i,Y}$, there exists a backward iterate $\varphi^Y_{-t}(x)$ which belongs either to some
$V_i$ or to $V_\mathrm{reg}$.
In both cases, $W^{ss}_{loc,Y}(\varphi^Y_{-t}(x))$ intersects $W^{u}_{Y}(\gamma_Y)$ transversely.
This concludes.
\end{proof}

\section{Density of the stable manifold}\label{s:stable-dense}

This section is devoted to the following result which will be used to prove the density of the stable manifold.

\begin{Theorem}\label{Thm:stable-dense}
There is a dense $G_\delta$ set ${\cal G}$ in ${\cal X}^1(M)$ such that for any $X\in{{\cal G}}$, for any singular hyperbolic transitive attractor $C(\sigma)$ of $X$,
and for any hyperbolic periodic orbit $\gamma$ in $C(\sigma)$, there exist a neighborhood $U$ of $C(\sigma)$ and a neighborhood ${\cal U}_X$ of $X$ with the following property.

For any $Y\in{\cal U}_X$, the stable manifold $W^s_Y(\gamma_Y)$ is dense in $U$.
Moreover, for any periodic orbit $\gamma'\subset C(\sigma_Y)$,
the set of transverse intersections $W^u_Y(\gamma')\pitchfork W^s_{Y}(\gamma_Y)$ is dense in $W^u_Y(\gamma')$.
\end{Theorem}

The set  ${\cal G}$ is the dense G$_\delta$ subset of vector fields satisfying the Items~\ref{i.KS} and~\ref{i.chain-homo} of Proposition~\ref{Pro:generic}.
In the whole section, we fix $X\in {\cal G}$ and $C(\sigma)$, $\gamma$ as in the statement of Theorem~\ref{Thm:stable-dense}.

\subsection{Center-unstable cone fields}\label{s.cones}
We consider the singular hyperbolic splitting $T_{C(\sigma)}M=E^{ss}\oplus E^{cu}$ on $C(\sigma)$
(see the Item~\ref{i.splitting-type-Y} of Proposition~\ref{Pro:strong-stable-Y}). These bundles may be extended as a continuous (but maybe not invariant)
splitting 
$T_{U_0}M=\widetilde E^{ss}\oplus \widetilde E^{cu}$ on a neighborhood $U_0$ of $C(\sigma)$.
For $x\in U_0$ and $\alpha>0$ we define the \emph{center-unstable cone}
$${\cal C}_{\alpha}(x)=\{v\in T_x M:~v=v^s+v^c,~v^s\in \widetilde E^{ss}(x),~v^c\in \widetilde E^{cu}(x),~\|v^c\|\le \alpha \|v^{s}\|\}.$$

Since the splitting is uniformly dominated, we have the following lemma. 
\begin{Lemma}\label{Lem:cone-for-tangent}
 For any $\alpha>\beta>0$, there are $T>0$, a $C^1$-neighborhood $\cU_1$ of $X$ and neighborhood $U_1\subset U_0$ of $C(\sigma)$  such that for any $Y\in\cU_1$ and for any orbit segment $\varphi^Y_{[0,t]}(x)\subset U_1$ with $t\ge T$,
$$D\varphi_t^Y({\cal C}_\alpha(x))\subset {\cal C}_{\beta}(\varphi_t^Y(x)).$$
\end{Lemma}
\begin{proof}
The proof for the vector field $X$ above the class $C(\sigma)$ is standard, hence is omitted.
The conclusion extends to neighborhoods of $X$ and $C(\sigma)$ by continuity of the cone fields and of $D\varphi^Y_t(x)$, $t\in [T,2T]$, with respect to $(x,t,Y)$.
\end{proof}

In the following we fix $\alpha>0$ small and set $\cC=\cC_a$. Note that (by increasing $T$, reducing $\cU_1$, $U_1$ and using the singular hyperbolicity)
the following property holds: for any $Y\in\cU_1$, any orbit segment $\varphi^Y_{[0,t]}(x)\subset U_1$ with $t\ge T$
and any $2$-dimensional subspace $P\subset {\cal C}_\alpha(x)$, 
\begin{equation}\label{e.cu-expansion}
|{\det}(D\varphi^Y_t)|_{P}|\ge 2.
\end{equation}
\medskip

The \emph{inner radius} $r$ of a submanifold $\Gamma$ is the supremum of $R>0$
such that the exponential map $B(0,R)\subset T_x\Gamma\to \Gamma$ with respect to the metric induced on $\Gamma$ by the Riemannian metric of $M$ is well defined and injective for some $x\in \Gamma$.
Note that $\Gamma$ always contain a submanifold $\Gamma'\subset \gamma$
with same inner radius $r$ and with diameter smaller than or equal to $2r$:
indeed one considers a sequence of balls $B(x_k,R_k)\subset \Gamma$ with $R_k\to R$
that are the images of exponential maps and consider a limit point $x$ for $(x_k)$;
the submanifold $\Gamma'$ is the ball $B(x,R)\subset \Gamma$.

Theorem~\ref{Thm:stable-dense} 
is a consequence of the next proposition.

\begin{Proposition}\label{Prop:stable-dense}
Under the setting of Section~\ref{s.cones},
there exists $\varepsilon>0$, a $C^1$-neighborhood $\cU_2\subset \cU_1$ of $X$
and a neighborhood $U_2$ of $C(\sigma)$
with the following property.

For any $Y\in \cU_2$ and any submanifold $\Gamma\subset U_2$
of dimension $\dim(E^{cu})$, tangent to the center-unstable cone field $\mathcal C$ and satisfying $Y(x)\in T_x\Gamma$ for each $x\in \Gamma$,
there is $t>0$ such that $\varphi^Y_{[0,t]}(\Gamma):=\cup_{s\in [0,t]}\varphi^Y_{s}(\Gamma)$ contains a submanifold
tangent to $\cC$, with dimension $\dim(E^{cu})$ and inner radius larger than $\varepsilon$.
\end{Proposition}
The proof of this proposition is postponed to the end of the Section~\ref{s:stable-dense}.

\begin{proof}[Proof of Theorem~\ref{Thm:stable-dense}]
Since $C(\sigma)$ is a transitive attractor, it is the chain-recurrence class of $\gamma$.
By Item~\ref{i.chain-homo} of Proposition~\ref{Pro:generic}, $C(\sigma)$ coincides with the homoclinic class $H(\gamma)$ of $\gamma$. Hence, the stable manifold of $\gamma$ is dense in $C(\sigma)$, hence
intersects transversally any submanifold of dimension $\dim(E^{cu})$ tangent to $\cC$
of inner radius $\varepsilon$ and which meets a small neighborhood $U_3\subset U_2$
of $C(\sigma)$. By continuation of the stable manifold,
this property is still satisfied for vector fields $Y$ in a small $C^1$-neighborhood $\cU_X\subset \cU_2$
of $X$. Since $C(\sigma)$ is an attractor, one may reduce $\cU_X$,
choose a smaller neighborhood $U\subset U_3$ of $C(\sigma_X)$
and assume that
$\varphi^Y_t(\overline U)\subset U_3$ for any $t>0$ and $Y\in \cU_X$.
Let $O$ be a small isolating neighborhood of ${\rm Sing}(X)$.
By the Item 2 of Proposition~\ref{Pro:strong-stable-Y}, $X$ is tangent to $\cC$ on $C(\sigma_X)\setminus O$.
Up to reduce the neighborhoods $U$ and $\cU_2$,
one can thus require that the vector fields $Y\in \cU_2$ are tangent to $\cC$ on $U\setminus O$.

In order to check that $W^s_Y(\gamma_Y)$ is dense in $U$
for any $Y\in \cU_X$,
we choose an arbitrary non-empty open subset $V\subset U\setminus O$ and a submanifold
$\Gamma\subset V$ as in the statement of Proposition~\ref{Prop:stable-dense}:
such a submanifold can be built by choosing a small disc $\Sigma$ of dimension $\dim(E^{cu})-1$ tangent to $\cC$
and transverse to $X$; then we set $\Gamma=\cup_{|t|<\delta}\varphi^Y_t(\Sigma)$ for $\delta$ close to $0$;
since $Y$ is tangent to $\cC$ on $U\setminus O$, the submanifold $\Gamma$ is tangent to $\cC$ as required.
From the choice of $U$ and the Proposition, there exists $t>0$ such that $\varphi^Y_t(\Gamma)$ has inner radius larger than $\varepsilon$,
intersects $U_3$, is still tangent to $\mathcal C$, hence meets $W^s(\gamma_Y)$ transversally. This shows that $W^s(\gamma_Y)$ intersects $V$ as required.
Let us now consider an open set $V\subset O$. Since $O$ is isolating for $Y$, there exists a forward iterate
$\varphi^Y_t(V)$ which meets $U\setminus O$. In particular $W^s_Y(\gamma_Y)$ meets $\varphi^Y_t(V)$, hence $V$.
 We have proved that $W^s_Y(\gamma_Y)$ is dense in $U$.

For the last part of the Theorem, one considers a small open subset of $W^u_Y(\gamma')$. It is a submanifold $\Gamma$
as considered above, hence it intersects transversally $W^s_Y(\gamma_Y)$.
\end{proof}

\subsection{Cross sections, holonomies}
Let $Y$ be a vector field whose singularities $\rho$ are hyperbolic.
One associates to each of them their local stable and unstable manifold $W^s_{Y,loc}(\rho), W^s_{Y,loc}(\rho)$.

\begin{Definition}\label{Def:cross}
A \emph{cross section} of $Y$ is a co-dimension $1$ sub-manifold  $D\subset M\setminus \mathrm{Sing}(Y)$
such that there exists $\alpha>0$ and a compact subset $\Delta\subset D$
satisfying:
\begin{itemize}
\item[--] $\overline D\cap \Sing(X)=\emptyset$ and for each $x\in D$, the angle between $Y(x)$ and $T_xD$ is larger than $\alpha$,
\item[--] the interior of $\Delta$ in $D$ intersects any 
forward orbit in $M\setminus \cup_\rho W^s_{Y,loc}(\rho)$ and any backward orbit in $M\setminus \cup_\rho W^u_{Y,loc}(\rho)$,
\item[--] For each singularity $\rho$,
$D$ intersects $W^s_{Y,loc}(\rho)$ along a fundamental domain contained in the interior of $\Delta$:
the intersection is a one-codimensional sphere inside $W^s_{Y,loc}(\rho)$
which meets exactly once each orbit in $W^s_{Y,loc}(\rho)$;
similarly $D$ intersects $W^u_{Y,loc}(\rho)$ along a fundamental domain contained in the interior of $\Delta$.
\end{itemize}
Such a compact set $\Delta\subset D$ is called a \emph{core} of the cross section $D$.
\end{Definition}

\begin{figure}
\begin{center}
\includegraphics[width=5cm,angle=0]{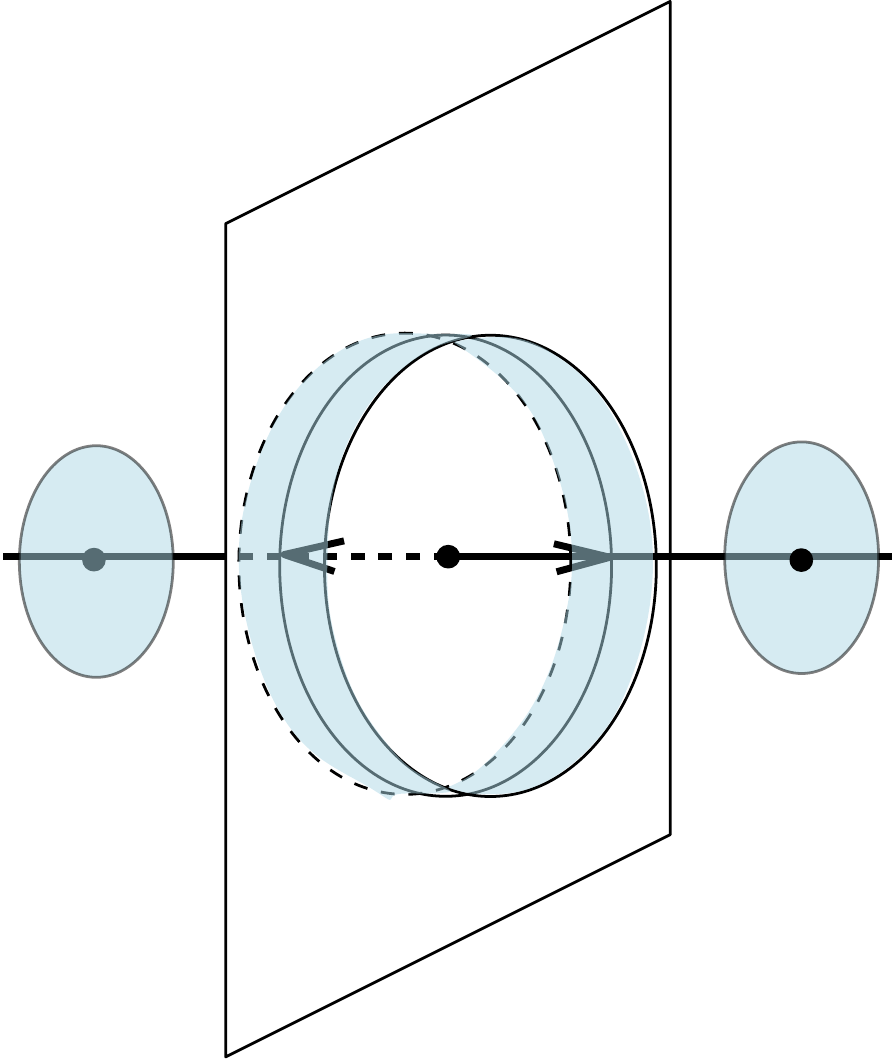}
\put(-75,70){\small $\rho$}
\put(-35,150){\small $W^s(\rho)$}
\end{center}
\caption{Construction of the cross-sections close to a singularity.\label{f.cross-section}}
\end{figure}

\begin{Lemma}\label{l.existence-cross-section}
There exists a cross section $D$ for $Y$.
\end{Lemma}
\begin{proof}
We cover the set of regular orbits by finitely small transverse sections $\Sigma$
such that each local stable or unstable manifold of a singularity $\rho\in C(\sigma)$
meets only one of these sections, along a one-codimensional sphere.
The union of these small sections satisfies the definition, but is maybe not a submanifold.
One can always perturb them so that each intersection $\Sigma_1\cap \Sigma_2$ is transverse.
We modify them so that they become pairwise disjoint:
each time two small sections $\Sigma_1,\Sigma_2$ intersect along a
submanifold $N$, one removes some small neighborhoods $U_1,U_2$
of $N$ in each of them
and one adds two new small disjoint sections which are the images of $U_1,U_2$
by small times of the flow.

Let $\Delta$ be the union of the obtained transverse sections.
The section $D$ is built by enlarging slightly $\Delta$.
\end{proof}

We now fix the cross section $D$ for $Y$.
Note that it still satisfies the definition of cross section (with the same core $\Delta$)
for vector fields $Z$ that are $C^1$-close to $Y$.
We also consider an open set $U\subset M$.

\begin{Definition}\label{Def:holonomy}
A \emph{holonomy} for $(Y,D, U)$ is a
$C^1$-diffeomorphism $\pi\colon V\to \tilde V$
between open sets $V,\tilde V\subset D\cap U$
such that
there exists a continuous function $t\colon V\to (0,+\infty)$ satisfying
for each $x\in V$:
\begin{itemize}
\item $\pi(x)=\varphi_{t(x)}(x)\in \tilde V$,
\item $\varphi_s(x)\in U$ for any $s\in [0,t]$.
\end{itemize}
The time $t$ is called \emph{transition time} of $x$ for $\pi$. \emph{(We do not require $t$ to be the first return time to $\tilde V$.
)}

\end{Definition}

\paragraph{\it Construction and continuation of holonomies.} One can build a holonomy by considering two open subsets $V_0,\tilde V_0\subset D$,
an open set $W\subset U$ and an open interval $I$ in $(0,+\infty)$ satisfying:
\begin{description}
\item[(H)] For any piece of orbit $\{\varphi^Y_s(x), s\in [0,t]\}$ in $W$ with $t\in I$,
at most one point $\varphi^Y_s(x)$ with $s\in [0,t]\cap I$ belongs to $\tilde V_0$.
\end{description}
One then considers all the connected pieces of orbit contained in $W$
which meet both $V_0$ and $\tilde V_0$ at points $x\in V_0$ and $\varphi^Y_t(x)\in \tilde V_0$ for some $t\in I$:
the set $V$ (resp. $\tilde V$) is the set of the points $x$ (resp. $\varphi^Y_t(x)$) and we set $\pi(x)=\varphi_t(x)$. By property (H), the map $\pi$ is well defined;
since the flow is transverse to $D$, the sets $V$ and $\tilde V$ are open in $D$.
If the property (H) still holds for vector fields $Z$ that are $C^1$-close to $Y$, this construction defines also holonomies
$\pi^Y\colon V^Z\to \tilde V^Z$ for $Z$, called \emph{continuation of $\pi$}.
\medskip

The dynamics can be covered by a finite collection of holonomies admitting continuations. Since we may want to localize the dynamics, one considers
an attracting invariant compact set $\Lambda$ with an \emph{attracting neighborhood} $U$: there exists $t_0>0$ such that the orbit
$(\varphi_{s}(x))_{s>t_0}$ of any point $x\in U$ is contained in $U$ and accumulates
on a subset of $\Lambda$.

\begin{Proposition}\label{p.holonomy}
Let us consider for $Y$ a cross section $D$ with a core $\Delta$,
an attracting set $\Lambda$ with an attracting neighborhood $U$
and $T'>0$.
Then, there exist a finite collection of holonomies
$\pi_i^Y\colon V_i^Y\to \tilde V_i^Y$, $i=1,\dots,\ell$
for $(Y,D,U)$, a $C^1$-neighborhood $\cU$ of $Y$,
a neighborhood $O$ of $\Lambda\cap \Delta$ in $D$,
and $\varepsilon_0>0$ such that for each $Z\in \cU$, 
\begin{itemize}
\item[(i)] the holonomies admit continuations
$\pi_i^Z\colon V_i^Y\to \tilde V_i^Z$,
\item[(ii)] $O\setminus \cup_\rho W^s_{loc}(\rho_Z)\subset \cup_i V_i^Z$
and $\cup_i \tilde V_i^Z\subset \Delta$,
\item[(iii)] the transition time of each honolomy $\pi_i^Z$ is bounded from below by $T'$,
\item[(iv)] if a set $A\subset D\setminus \cup_\rho W^s_{loc}(\rho_Z)$
is contained in the $\varepsilon_0$-neighborhood of $\Lambda\cap \Delta$ and
has diameter smaller than $\varepsilon_0$, then $A$ is contained in a $V_i^Z$.
\end{itemize}


\end{Proposition}
\begin{proof}

The holonomies are built in two different ways depending if the points are close to $W^s_{loc}(\mathrm{Sing}(Y))$ or not.

\begin{Claim}
For any $\rho\in \Lambda\cap \mathrm{Sing}(Y)$,
there exist a neighborhood $O_\rho$ of $W^s_{loc}(\rho)\cap \Lambda\cap \Delta$ in $D$
and a holonomy $\pi$ admitting continuations for any $Z$ close to $Y$,
such that the domain $V^Z$ of $\pi^Z$ contains $O_\rho\setminus W^s_{loc}(\rho^Z)$,
the image $\tilde V^Z$ is contained in $\Delta$
and the transition time is bounded from below by $T'$.
\end{Claim}
\begin{proof}
By definition of the cross section,
the sets $\Delta^s=W^s_{loc}(\rho)\cap \Delta$ and
$\Delta^u=W^u_{loc}(\rho)\cap \Delta$ are compact fundamental domains in the local stable and unstable manifolds.
Let $V_0$ and $\tilde V_0$ be neighborhoods
of $W^s_{loc}(\rho)\cap \Lambda\cap \Delta$
and $W^u_{loc}(\rho)\cap \Lambda\cap \Delta$
in $D$ and let $W\subset U$ be a small neighborhood of $\Lambda\cap (W^s_{loc}(\rho)\cup W^u_{loc}(\rho))$.
A piece of orbit in $W$ intersects $\tilde V_0$ at most one, hence the property (H)
is satisfied for the time interval $I=(0,+\infty)$ and any $Z$ close to $Y$.
As a consequence, this defines a holonomy $\pi\colon V\to \tilde V$ which admits a continuation
for vector fields $Z$ that are $C^1$-close to $Y$.

For any $Z$ close to $Y$, the intersection $\Delta\cap W^s_{loc}(\rho_Z)$ is contained in a small neighborhood
$O_\rho$ of $\Delta^s$ in $D$ and $V^Z$ contains $O_\rho\setminus W^s_{loc}(\rho_Z)$. Moreover by definition of the cross sections,
$W^u_{loc}(\rho)\cap \Lambda\cap \Delta$ is contained in the interior of $\Delta$
in $D$, hence $\tilde V\subset \Delta$.
We may have chosen $W$ small enough so that the transition times
of the holonomies are larger than $T'$.
\end{proof}

\begin{Claim}
For each $x\in (\Lambda\cap \Delta)\setminus \cup_\rho W^s_{loc}(\rho)$,
there exist a neighborhood $O_x$ of $x$ in $D$ and a holonomy $\pi$ admitting continuations for any $Z$ close to $Y$,
such that
the domain $V^Z$ of $\pi^Z$ contains $O_x$, the image $\tilde V^Z$
is contained in $\Delta$ and the transition time is bounded from below by $T'$.
\end{Claim}
\begin{proof}
By definition of the cross sections,
there exists $t>T'$ such that $\varphi_t(x)$ belongs to the interior of $\Delta$.
We choose small open neighborhoods $V_0\subset D$ and $\tilde V_0\subset \Delta$
in $D$, a small neighborhood $W\subset U$ of the piece of orbit
$\{\varphi_s(x), s\in [0,t]\}\subset \Lambda$ and the time interval $(t-\delta,t+\delta)$ for $\delta>0$ small.
This defines a holonomy, admitting a continuation for $Z$ close to $Y$
with transition time bounded from below by $T'$.
%
%
\end{proof}

By compactness of $\Lambda\cap \Delta$, we can select from the holonomies provided by the previous claims
a finite number of them
such that the union $O$ of the open sets $O_\rho$ and $O_x$
contain $\Lambda\cap \Delta$.
These holonomies admit continuations for vector fields $Z$ in a $C^1$-neighborhood
$\cU$ of $Y$ and satisfy the items (i), (ii) and (iii).
There exists $\varepsilon_0$ such that if $A\subset D$ has diameter
smaller than $\varepsilon_0$ and is contained in the $\varepsilon_0$-neighborhood
of $\Lambda\cap \Delta$, then $A$ is included in some open set
$O_\rho$ or $O_x$; the item (iv) follows.
\end{proof}

As in Section~\ref{s.cones},
we consider a center-unstable cone field on an open neighborhood $U_1$
of a singular hyperbolic chain-recurrence class $C(\sigma)$.
\begin{Definition} Let $Z$ be a vector field $C^1$-close to $Y$.
A \emph{cu-section} of $Z$  is a submanifold $N\subset D$ with dimension $\dim(E^{cu})-1$
such that $T_xN\oplus Z(x)\subset \cC(x)$ for each $x\in N$.
\end{Definition}

The forward invariance of the cone field (see Lemma~\ref{Lem:cone-for-tangent}) implies:

\begin{Lemma}\label{l.invariance-section}
Let us consider open sets $V_0,\tilde V_0$ and $W\subset U_1$ satisfying (H) and
the associated holonomy $\pi$. Then the image by $\pi$ of a cu-section $N$ is still a cu-section.
\end{Lemma}
\medskip
The minimal norm of a linear map $A$ between two euclidean spaces is defined by
$m(A):=\inf\{\|Av\|: \;\text{unit vector $v$}\}$.

\begin{Definition} A holonomy $\pi\colon V\to \tilde V$ is \emph{$10$-expanding} if there exists $\chi>10$
such that for any cu-section $N\subset V$, the derivative $D\pi|_N$ has minimal norm larger than $\chi$
with respect to the induced metrics on $N$ and $\pi(N)$.
\end{Definition}

The definition of the singular hyperbolicity (in particular condition~\eqref{e.cu-expansion})
and the uniform transversality between $D$ and the vector field imply:
\begin{Lemma}\label{l.expanding}
Let us consider a cross section $D$ and an open neighborhood $U_1$  of
a singular hyperbolic chain-recurrence class $C(\sigma)$ as in Section~\ref{s.cones}.
There exists $T'>0$ such that any holonomy $\pi\colon V\to\widetilde V$
for $(Y,D,U_1)$, whose transition times are bounded below by $T'$, is $10$-expanding.
\end{Lemma}
\medskip

\subsection{Proof of the Proposition~\ref{Prop:stable-dense}}
Let us consider the neighborhoods $U_1$ of $C(\sigma_X)$,
$\cU_1$ of $X$ and the cone field $\cC$ satisfying condition~\eqref{e.cu-expansion}.
Let us consider a cross section $D$ for $X$ with a core $\Delta\subset D$ (as given by Lemma~\ref{l.existence-cross-section}).
One can always replace $\cC$ and $U_1$ by forward iterates under $\varphi^X$, so that
$\cC$ is arbitrarily close to the bundle $E^{cu}$ on $C(\sigma_X)$.

\begin{Claim}
(Up to replace $\cC$ and $U_1$ by forward iterates), for each singularity $\rho$ one can assume that
$\cC$ is transverse to
$D\cap W^s_{loc}(\rho)$ at points of $C(\sigma)$.
\end{Claim}
\begin{proof}
For each singularity $\rho$, the intersection $D\cap W^s_{loc}(\rho)$ is a one-codimensional
sphere in $W^s_{loc}(\rho)$ transverse to $X$.
Since $E^{cu}\cap TW^s_{loc}(\rho)= \RR X$ at regular points of $C(\sigma)$,
one deduces that at points of $D\cap W^s_{loc}(\rho)\cap C(\sigma)$, the submanifold
$D\cap W^s_{loc}(\rho)$ is transverse to $E^{cu}$, hence to $\cC$.
\end{proof}

Since $C(\sigma_X)$ is an attractor,
it admits a neighborhood $U\subset U_1$ which is attracting
such that $\Lambda:=C(\sigma_X)$ is the maximal invariant set in $U$.
By the claim above, up to reduce the neighborhoods $U$ and $\cU_1$, one can require
that for any $Y\in \cU_1$ the cone field $\cC$ is transverse to
$D\cap W^s_{loc}(\rho_Y)\cap U$ for each singularity $\rho_Y\in U$ of $Y$.
In particular, there exists $\varepsilon_1>0$ such that
any cu-section $N$ of $Y\in \cU_2$ with diameter smaller than $\varepsilon_1$
can intersect $W^s_{loc}(\rho_Y)$ in at most one point.

Let $T'>0$ be given by Lemma~\ref{l.expanding}.
By Proposition~\ref{p.holonomy}, there exists a finite collection
of holonomies $\pi_i^Y\colon V_i^Y\to \tilde V_i^Y$, $i=1,\dots,\ell$,
defined for any vector field $Y$ in a neighborhood $\cU_2\subset \cU_1$ of $X$
and whose transition times are bounded from below by $T'$.
The union $\cup_i V_i^Y$ covers a uniform neighborhood $O$ of $C(\sigma_X)\cap \Delta$.
Since $C(\sigma_X)$ is an attractor for $\varphi^X$ and by our choice of the cross section $D$ and of $\Delta$,
one can reduce the neighborhood $\cU_2$ of $X$ and choose a neighborhood
$U_2\subset U$ of $C(\sigma_X)$ such that
for any $Y\in \cU_2$ and any $x\in U_2$:
\begin{itemize}
\item[--] $x$ belongs to the stable manifold of a singularity $\rho_Y$ or has a forward iterate
by $\varphi^Y$ which belongs to the interior of $\Delta$ in $D$,
\item[--] the forward orbit of $x$ is contained in $U$.
\end{itemize}
Moreover there exists $\varepsilon_0$ satisfying the item (iv) of Proposition~\ref{p.holonomy}.
We can always reduce $\varepsilon_0$ to be smaller than $\varepsilon_1$.

The angle between the cross section $D$ and the vector field is bounded away from zero by $\alpha$.
Hence if $\varepsilon_0$ is reduced enough and $\varepsilon>0$ is chosen small, then
for any cu-section $N\subset \Delta$ with inner radius $\varepsilon_0$ and diameter less than $3\varepsilon_0$,
the set $\cup\{\varphi_s^Y(N):\; |s|\leq \varepsilon_0\}$ is a submanifold with inner diameter larger than $\varepsilon$.
\medskip

Let us consider $Y\in \cU_2$ and $\Gamma$ be a submanifold of dimension $\dim(E^{cu})$ as in the statement of Proposition~\ref{Prop:stable-dense}.
Note that $C(\sigma_Y)$ is not a sink (it contains $\gamma_Y$), hence $\sigma_Y$ has a non-trivial unstable space.
From the Item~\ref{i.strong-stable-outside-Y} of Proposition~\ref{Pro:strong-stable-Y}, this implies that the bundle $E^{cu}$ has dimension at least $2$.
From the area expansion~\eqref{e.cu-expansion}, the stable manifolds of the singularities are meager in $\Gamma$.
By definition of the cross section and of its core $\Delta$,
one can thus find $x\in \Gamma$ having a forward iterate $\varphi^Y_{t_0}(x)$ in $\Delta\setminus \cup_\rho W^s(\rho_Y)$.
By construction and forward invariance the set $\Gamma':=\cup_{|t-t_0|<\delta}\; \varphi_t^Y(\Gamma)$ is still a submanifold tangent to $\cC$ if $\delta>0$ is small enough. Since $Y(z)\in T_z\Gamma'$,
the intersection $N:=\Gamma'\cap D$ is a cu-section.
See Figure~\ref{f.singularity}

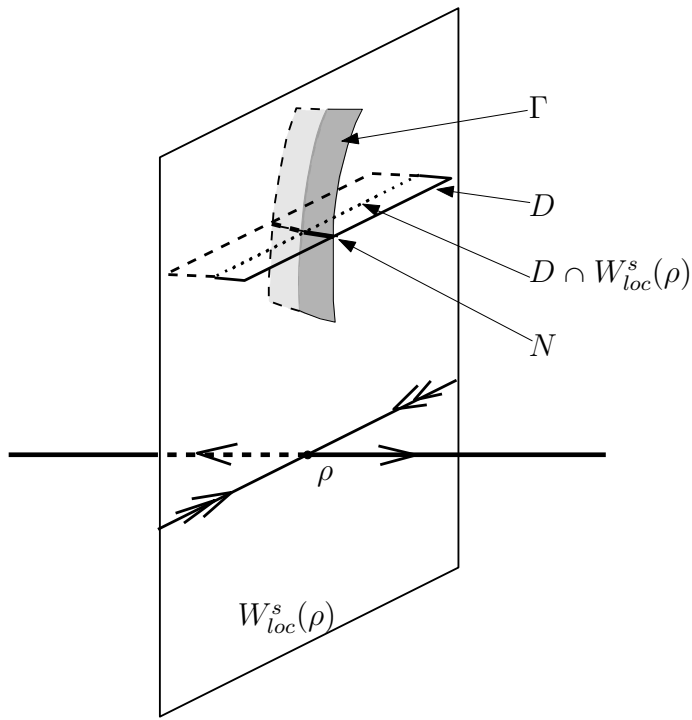
\begin{figure}
\begin{center}
\includegraphics[width=8cm,angle=0]{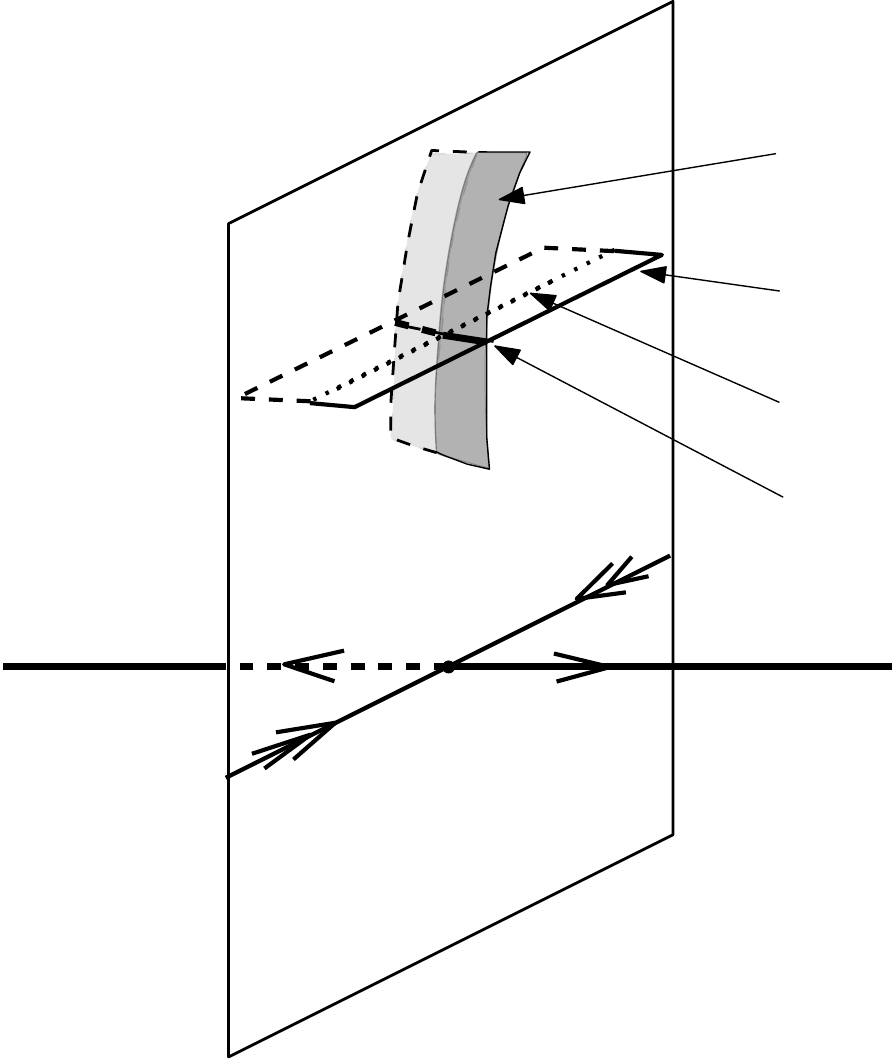}
\put(-110,90){$\rho$}
\put(-140,35){$W^s_{loc}(\rho)$}
\put(-30,227){$\Gamma$}
\put(-30,190){$D$}
\put(-30,137){$N$}
\put(-30,165){$D\cap W^s_{loc}(\rho)$}
\end{center}
\caption{Construction of the cu-section $N$.\label{f.singularity}}
\end{figure}

We inductively build a sequence of cu-sections $N_n\subset D\setminus \cup W^s_{loc}(\rho_Y)$
with diameter smaller than $\varepsilon_0$ and contained in the orbit $\cup_{t>0}\varphi^Y_t(N)$ of $N$. We denote by $r_n$
their inner radius.
As explained in Section~\ref{s.cones}, we can reduce $N_n$ without reducing the inner radius
in a such a way that the diameter of $N_n$ is smaller than $3r_n$.

Since by definition a cu-section is transverse to $W^s(\rho_Y)$,
one can choose $N_0\subset N\setminus W^s_{loc}(\rho_Y)$.
The cu-section $N_n$ is then built inductively from $N_{n-1}$ by:
\begin{itemize}
\item[a-] Choosing a domain $V_i^Y$ intersecting $N_{n-1}$:
if $N_{n-1}$ is disjoint from $W^s_{loc}(\rho_Y)$ for any singularity $\rho_Y$,
then by Item (iv) of Proposition~\ref{p.holonomy}, there exists a domain $V_i^Y$ which contains $N_{n-1}$;
otherwise, there exists a singularity $\rho_Y$ such that $N_{n-1}\setminus W^s_{loc}(\rho_Y)$
is contained in a domain $V_i^Y$.

\begin{Claim}
The inner diameter of $N_{n-1}\cap V_i^Y$
is larger than $r_{n-1}/3$.
\end{Claim}
\begin{proof}
If $N_{n-1}\subset V_i^Y$, there is nothing to prove.
Otherwise by our choice of $\varepsilon_1>\varepsilon_0$,
the intersection $N_{n-1}\cap W^s_{loc}(\rho_Y)$ contains only one point.
Hence the inner radius of $N_{n-1}\setminus W^s_{loc}(\rho_Y)$
is larger than one third of the inner radius of $N_{n-1}$.
See Figure~\ref{f.section}.
\end{proof}
\begin{figure}
\begin{center}
\includegraphics[width=5cm,angle=0]{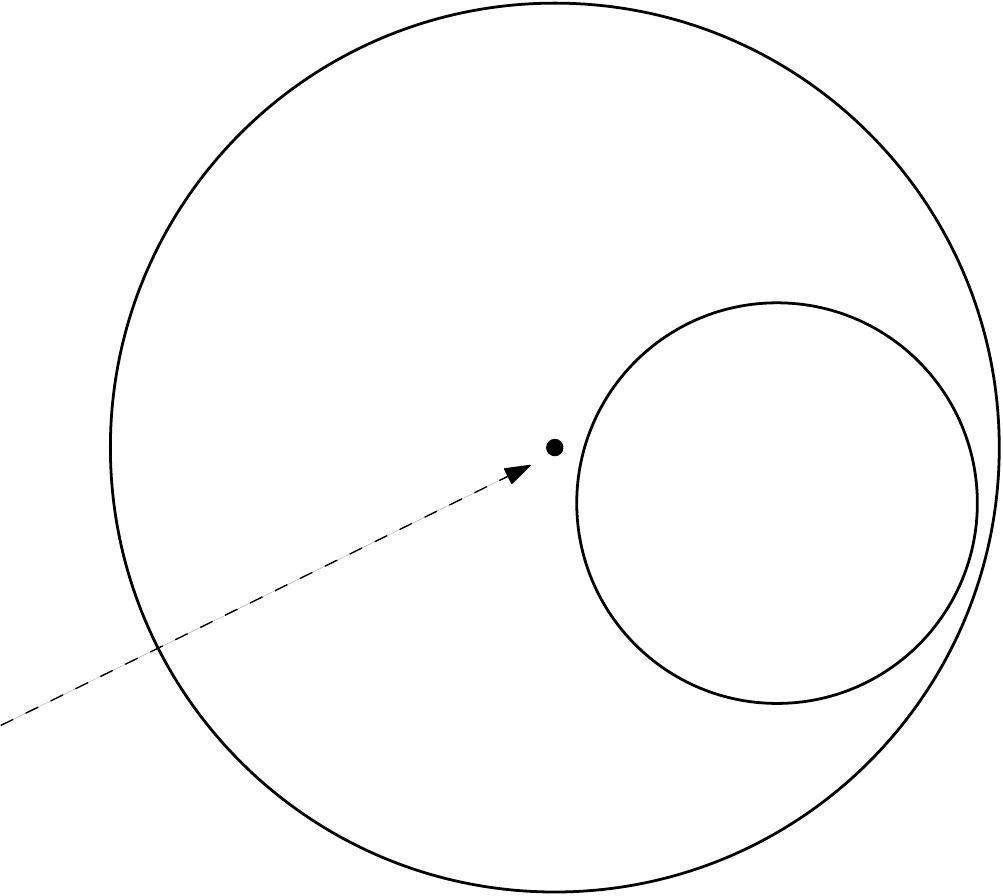}
\put(-95,90){\small $N_{n-1}$}
\put(-50,45){\small $\pi^{-1}_i(N_n)$}
\put(-220,20){\small  $N_{n-1}\cap W^s_{loc}(\rho)$}
\end{center}
\caption{Construction of $N_n$ from $N_{n-1}$.\label{f.section}}
\end{figure}

\item[b-] Considering the image $A_n:=\pi_i^Y(N_{n-1}\cap V_i^Y)$
and the inner radius $r_n$ of $A_n$.
\item[c-] Choosing $N_n$ as a subset of $A_n$ having the same inner radius as $A_n$
and satisfying $\diam(V_n)<3r_n$ (as explained in Section~\ref{s.cones}).
By Lemma~\ref{l.invariance-section}, $N_n$ is a cu-section.
\end{itemize}
By Lemma~\ref{l.expanding}, the holonomies are $10$-expanding,
hence any ball in $N_{n-1}\cap V_i^Y$ has radius smaller than $r_n/10$.
This gives $r_{n-1}/3<r_n/10$.
This implies that the sequence of radii increases until $r_n\geq \varepsilon_0$
(and then the construction stops).

By our choice of $\varepsilon,\varepsilon_0$,
the set $\cup\{\varphi_s^Y(N_n):\; |s|\leq \varepsilon_0\}$ contains a submanifold tangent to $\cC$
with inner diameter larger than $\varepsilon>0$.
This ends the proof of the Proposition~\ref{Prop:stable-dense} and of Theorem~\ref{Thm:stable-dense}.
\qed

\section{The proofs of main theorems}

We will first prove Theorem~\ref{Thm:main} in a generic setting.

\begin{Theorem A'}
There is a dense $G_\delta$ set $\cG\in{\cal X}^1(M)$ such that for any $X\in\cG$, any singular hyperbolic Lyapunov stable chain-recurrence class $C(\sigma)$ of $X$ is robustly transitive.

More precisely, there are neighborhoods ${\cal U}_X$ of $X$ and $U$ of $C(\sigma)$
such that the maximal invariant set of $U$ for any $Y\in{\cal U}_X$ is an attractor which coincides with $C(\sigma_Y)$.
If it is not an isolated singularity, it is a homoclinic class.

\end{Theorem A'}

\begin{proof}
We consider a dense G$_\delta$ set $\cG\in{\cal X}^1(M)$ whose elements satisfy the conclusions
of Proposition~\ref{Pro:generic} and Theorems~\ref{Thm:periodic-contained}, \ref{Thm:unstable-dense}, \ref{Thm:stable-dense}.

By Theorem~\ref{Thm:periodic-contained}, we know that $C(\sigma)$ is in fact an attractor for $\varphi^X$, hence it is locally maximal.
By Item~\ref{i.isolated} of Proposition~\ref{Pro:generic}, $C(\sigma)$ is robustly locally maximal, i.e. there are a neighborhood $\cU_X$ of $X$ and a neighborhood $U$ of $C(\sigma)$ such that $C(\sigma_Y)$ is the maximal invariant set in $U$ for any $Y\in\cU_X$.
Since $C(\sigma)$ is an attractor, we can assume that $\varphi_T^X(\overline U)\subset U$ for some $T>0$
and that the same holds for any $Y$ close to $X$. As a consequence, up to reduce $\cU_X$ and $U$,
the class $C(\sigma_Y)$ is also an attractor for $Y\in\cU_X$.
The remaining is to prove that $C(\sigma_Y)$ is a homoclinic class (if $C(\sigma)$ is not trivial).

By Theorem~\ref{Thm:periodic-contained}, and up to reduce $\cU_X$, there exists a periodic orbit
$\gamma\subset C(\sigma)$ such that $\gamma_Y\subset C(\sigma_Y)$ for any $Y\in \cU_X$.
Furthermore, the conclusions of Theorems~\ref{Thm:unstable-dense} and~\ref{Thm:stable-dense} hold.
Consider a point $x\in C(\sigma_Y)$. We will show that $x$ is accumulated by transverse homoclinic points of $\gamma_Y$.

\begin{Claim}
Any neighborhood $V_x$ of $x$ intersects $W^{u}_{Y}(\gamma_Y)$.

\end{Claim}
\begin{proof}[Proof of the Claim.]
If $\alpha(x)$ is not a single singularity, then by Theorem~\ref{Thm:unstable-dense},
the unstable manifold $W^{u}_{Y}(\gamma_Y)$ intersects $W^{ss}_{loc,Y}(\varphi^Y_{-t}(x))$ for any $t>0$.
By choosing $t>0$ large, one deduces that $W^{u}_{Y}(\gamma_Y)$ intersects $W^{ss}_{loc,Y}(x)$
at a point arbitrarily close to $x$.

Let us assume now that $\alpha(x)$ is a single singularity $\rho_Y$.
By the Item~\ref{i.transverse-close} of Proposition~\ref{Pro:generic},
the unstable manifold of $\gamma_Y$ intersect $W^{s}_{Y}(\rho_Y)$ transversely. Hence by the inclination lemma, there is a point $y\in V_x\cap W^{u}_{Y}(\gamma_Y)$.
The claim is verified in both cases.
\end{proof}

By Theorem~\ref{Thm:stable-dense}, the transverse intersection points
between $W^{u}_{Y}(\gamma_Y)$ and $W^{s}_{Y}(\gamma_Y)$ are dense in $W^{u}_{Y}(\gamma_Y)$.
This concludes that $C(\sigma_Y)$ is a homoclinic class. Since a homoclinic class is transitive, we have that $C(\sigma_Y)$ is transitive.
\end{proof}

We can now conclude the proof of the main results.

\begin{proof}[Proof of Theorems~\ref{Thm:main} and \ref{Thm:homoclinic-related}]
We first notice that these theorems hold for singular hyperbolic attractors
which do not contain any singularity and for isolated hyperbolic singularity:
in these cases the attractors are uniformly hyperbolic and the proof is classical.
In the following we only consider non-trivial classes which contain at least singularity.

%
%
Let $\cG$ be a dense $G_\delta$ subset of ${\cal X}^1(M)$ satisfying the conclusions of Theorem A',
Propositions~\ref{Pro:robust}, \ref{Pro:generic} and Theorems~\ref{Thm:unstable-dense}, \ref{Thm:stable-dense}.
For $X\in \cG$, the singularities are hyperbolic and finite
and there exists a neighborhood $\cU_X$ where the singularities admit a continuation,
satisfy Proposition~\ref{Pro:robust} and such that (as in Theorem A',
\ref{Thm:unstable-dense} and \ref{Thm:stable-dense})
the following property holds:
if $C(\sigma)$ is a (non-trivial) Lyapunov stable chain-recurrence class of a singularity $\sigma$ for $X$, then
there exists a periodic orbit $\gamma\subset C(\sigma)$ such that for any $Y\in \cU_X$:
\begin{itemize}
\item[--] the continuation $C(\sigma_Y)$ is a transitive attractor  and coincides with the homoclinic class $H(\gamma_Y)$,
\item[--] for any $x\in C(\sigma_Y)$ which does not belong to the unstable manifold of a singularity,
there exists a transverse intersection between $W^{ss}_Y(x)$ and $W^u_Y(\gamma_Y)$.
\item[--]  for any periodic orbit $\gamma'_Y\in C(\sigma_Y)$, there exists a transverse intersection between
$W^s_Y(\gamma_Y)$ and $W^u_Y(\gamma'_Y)$.
\end{itemize}
%
%
%
We define the dense open set
$$\cU=\bigcup_{X\in\cG}\cU_X.$$
Now we will verify that Theorem~\ref{Thm:main} holds in this open dense set $\cU$.

Take $Y\in\cU$, there is a vector field $X\in\cG$ such that $Y\in\cU_X$.
Consider a singular hyperbolic Lyapunov stable chain-recurrence class $C(\sigma_Y)$ of $Y$: it has the property that $W^u_Y(\sigma_Y)\subset C(\sigma_Y)$.
Hence by the first property of Proposition~\ref{Pro:robust}, we have that $W^u_X(\sigma_X)\subset C(\sigma_X)$.
Hence $C(\sigma_X)$ is Lyapunov stable by Item~\ref{i.unstable-Lyapunovstable} of Proposition~\ref{Pro:generic}.
By Theorem A', $C(\sigma_Y)$ is a robustly transitive attractor since $Y\in\cU_X$. Hence Theorem~\ref{Thm:main} holds.

Since $Y\in\cU_X$,
we also conclude that $C(\sigma_Y)=H(\gamma_Y)$ is a homoclinic class
and contain a dense subset of periodic points.
Moreover, for any periodic orbit $\gamma'_Y\subset C(\sigma_Y)$,
there is a transverse intersection between
$W^s_Y(\gamma_Y)$ and $W^u_Y(\gamma'_Y)$.
Considering any $x\in \gamma'_Y$, by Theorem~\ref{Thm:unstable-dense}, there also exists a transverse intersection between
$W^{ss}_Y(x)$ and $W^u_Y(\gamma_Y)$ because that $x$ is not contained in the unstable manifolds of singularities. The periodic orbits $\gamma_Y$ and $\gamma'_Y$ are thus
homoclinically related. Beeing homoclinically related is a transitive relation on hyperbolic periodic orbits
(by the inclination lemma), hence any two periodic orbits in $C(\sigma_Y)$
are homoclinically related. This concludes the proof of Theorem~\ref{Thm:homoclinic-related}.
\end{proof}

%
%
%
\begin{proof}[Proof of Corollary~\ref{Cor:dimension3}]
Now we assume that $\dim M=3$. We introduce the same open dense set $\cU$ as the one built in the proof of
Theorems~\ref{Thm:main} and \ref{Thm:homoclinic-related}.
Consider $Y\in\cU$ and a singular hyperbolic chain-recurrence class $C$.
If it is an isolated singularity or if it does not contain any singularity, then it is a uniformly hyperbolic set. Hence it is robustly transitive and a homoclinic class (by the shadowing lemma).

We can now assume that $C$ is non-trivial (i.e. it is not reduced to a single critical element) and that it contains a singularity $\sigma_Y$.
Since $\dim M=3$ and $C$ is non-trivial, the stable dimension of $\sigma_Y$ is $2$ or $1$. Assume for instance that it is $2$
(when it is $1$ we replace $X$ by $-X$).
By construction of $\cU$, there is a generic vector field $X\in \cG$ with a neighborhood $\cU_X$ such that $Y\in \cU_X$.
Since $C=C(\sigma_Y)$ is non-trivial,
by Proposition~\ref{Pro:robust}, item~\ref{Item:non-trivial---}, $C(\sigma_X)$ is non-trivial. Hence by the Item~\ref{i.dim3} of Proposition~\ref{Pro:generic},
$C(\sigma_X)$ is Lyapunov stable.
By definition of $\cU_X$ (introduced in the proof of Theorems~\ref{Thm:main} and \ref{Thm:homoclinic-related}),
the class $C(\sigma_Y)$ is a homoclinic class (hence transitive) for any $Y$ that is $C^1$-close to $X$.
%
%
%
%
%
%
\end{proof}

\vskip 5pt

\begin{tabular}{l l l}
\emph{\normalsize Sylvain Crovisier}
& \quad\quad \quad &
\emph{\normalsize Dawei Yang}
\medskip\\

\small Laboratoire de Math\'ematiques d'Orsay
&& \small School of Mathematical Sciences\\
\small CNRS - Universit\'e Paris-Sud
&& \small Soochow University\\
\small Orsay 91405, France
&& \small Suzhou, 215006, P.R. China\\
\small \texttt{Sylvain.Crovisier@math.u-psud.fr}
&& \small \texttt{yangdw1981@gmail.com}\\
&& \small \texttt{yangdw@suda.edu.cn}
\end{tabular}

\end{document}